\documentclass[11pt]{amsart}
\usepackage{xcolor}

\include{package}
\definecolor{grn}{rgb}{0,0.6,0}
\definecolor{mrn}{rgb}{0.3,0,0}
\definecolor{blue}{rgb}{0,0,0.7}
\definecolor{Mygray}{rgb}{0.75,0.75,0.75}
\definecolor{auburn}{rgb}{0.43, 0.21, 0.1}
\definecolor{britishracinggreen}{rgb}{0.0, 0.26, 0.15}
\definecolor{taupe}{rgb}{0.28, 0.24, 0.2}
\usepackage{amsmath,amssymb}
\newtheorem{theorem}{Theorem}

\newtheorem{propn}{Proposition}
\newtheorem{corollary}{Corollary}
\newtheorem{lemma}{Lemma}
\theoremstyle{definition}
\newtheorem{defn}{Definition}

\usepackage{latexsym,amssymb}
\begin{document}
\baselineskip=14.5pt
\title[Totally real bi-quadratic fields with large P\'{o}lya groups]{Totally real bi-quadratic fields with large P\'{o}lya groups} 

\author{Jaitra Chattopadhyay and Anupam Saikia}
\address[Jaitra Chattopadhyay]{Department of Mathematics, Indian Institute of Technology, Guwahati, Guwahati - 781039, Assam, India}
\email[Jaitra Chattopadhyay]{jaitra@iitg.ac.in, chat.jaitra@gmail.com}
\address[Anupam Saikia]{Department of Mathematics, Indian Institute of Technology, Guwahati, Guwahati - 781039, Assam, India}
\email[Anupam Saikia]{a.saikia@iitg.ac.in}

\begin{abstract}
For an algebraic number field $K$ with ring of integers $\mathcal{O}_{K}$, an important subgroup of the ideal class group $Cl_{K}$ is the {\it P\'{o}lya group}, denoted by $Po(K)$, which measures the failure of the $\mathcal{O}_{K}$-module $Int(\mathcal{O}_{K})$ of integer-valued polynomials on $\mathcal{O}_{K}$ from admitting a regular basis. In this paper, we prove that for any integer $n \geq 2$, there are infinitely many totally real bi-quadratic fields $K$ with $|Po(K)| = 2^{n}$. In fact, we explicitly construct such an infinite family of number fields. This extends an infinite family of bi-quadratic fields with P\'{o}lya group $\mathbb{Z}/2\mathbb{Z}$ given by the authors in \cite{self-ja}. This also provides an infinite family of bi-quadratic fields with class numbers divisible by $2^{n}$.
\end{abstract}

\renewcommand{\thefootnote}{}



\renewcommand{\thefootnote}{\arabic{footnote}}
\setcounter{footnote}{0}

\maketitle

\section{Introduction}

Let $K$ be an algebraic number field with ring of integers $\mathcal{O}_{K}$ and ideal class group $Cl_{K}$. Let $${\rm{Int}}(\mathcal{O}_{K}) = \{f(X) \in K[X] : f(\alpha) \in \mathcal{O}_{K} \mbox{ for all } \alpha \in \mathcal{O}_{K}\}.$$

It is a well-known fact (cf. \cite{cahen-chabert-book}, \cite{what-you}) that ${\rm{Int}}(\mathcal{O}_{K})$ is a free $\mathcal{O}_{K}$-module. An $\mathcal{O}_{K}$-basis $\{f_{i}\}_{i \geq 0}$ of ${\rm{Int}}(\mathcal{O}_{K})$ is said to be a regular basis if ${\rm{degree}}(f_{i}) = i$ for each $i \geq 0$. The existence of a regular basis leads to the definition of a P\'{o}lya field as follows.

\begin{defn} (cf. \cite{cahen-chabert-book}, \cite{zantema})
An algebraic number field $K$ is said to be a P\'{o}lya field if the $\mathcal{O}_{K}$-module ${\rm{Int}}(\mathcal{O}_{K})$ has a regular basis.
\end{defn}

It is quite interesting that the existence of a regular basis for $K$ is closely governed by the structure of $Cl_{K}$. To describe the interplay between $Cl_{K}$ and ${\rm{Int}}(\mathcal{O}_{K})$, we first introduce the following notation.

$$
\displaystyle\Pi_{q}(K) = \left\{\begin{array}{ll}
\displaystyle\prod_{\substack {\mathfrak{p} \in {\rm{Spec}}(\mathcal{O}_{K})\\ N \mathfrak{p} = q}} \mathfrak{p} &\mbox{; if } q \mbox{ is a prime power }\\
\mathcal{O}_{K} &\mbox{; otherwise}.
\end{array}\right.$$ 

\begin{defn} (cf. \cite{cahen-chabert-book}, \cite{ostrowski}, \cite{polya} \cite{zantema})\label{defn-2}
The P\'{o}lya group $Po(K)$ of $K$ is defined to be the subgroup generated by the ideal classes $[\Pi_{q}(K)]$ in $Cl_{K}$.
\end{defn}

It is known (cf. \cite{cahen-chabert-book}) that $K$ is a P\'{o}lya field if and only if $Po(K)$ trivial. Thus the study of ${\rm{Int}}(\mathcal{O}_{K})$ is translated into the study of $Cl_{K}$. Moreover, when $K/\mathbb{Q}$ is a finite Galois extension, it immediately follows from Definition \ref{defn-2} that $Po(K)$ is generated precisely by the ideal classes $[\displaystyle\Pi_{q}(K)]$, where $q$ runs through the norms of ramified primes. Therefore, there is a close connection between the ramified primes in $K$ and the P\'{o}lya group $Po(K)$. In particular, for quadratic fields, the following result is known due to the work of Hilbert.

\begin{propn} \cite{hilbert}\label{hilbert}
Let $K = \mathbb{Q}(\sqrt{d})$ be a quadratic field. Let $\varepsilon$ be the fundamental unit of $K$ and let $r_{K}$ denote the number of ramified primes in $K/\mathbb{Q}$. Then
$$
|Po(K)| = \left\{\begin{array}{ll}
2^{r_{K} - 2} &\mbox{; if } d > 0 \mbox{ and } N(\varepsilon) = 1\\
2^{r_{K} - 1} &\mbox{; otherwise}.
\end{array}\right.
$$
\end{propn}

From Proposition \ref{hilbert}, it is clear that the order of the P\'{o}lya group for a quadratic field is necessarily $1$ or a power of $2$. In \cite{zantema}, Zantema generalized this to finite Galois extensions $K/\mathbb{Q}$ and recently, Maarefparvar and Rajaei extended the result for relative extensions of number fields in \cite{relative}. Leriche addressed the cubic field case in \cite{leriche}. There has been a lot of studies on the P\'{o}lya group of bi-quadratic fields in recent times (cf. \cite{self-ja},\cite{rajaei-jnt},\cite{rajaei},\cite{abbas-hung},\cite{abbas-jnt}). In \cite{self-ja}, the authors determined the P\'{o}lya groups of bi-quadratic fields that are known to have an ``Euclidean ideal class". The interested reader is encouraged to look into \cite{self-jnt} and the references listed therein to know about Euclidean ideal class.

\medskip

In this paper, using Zantema's result, we explicitly construct an infinite family of totally real bi-quadratic fields with P\'{o}lya group having $2$-rank $n$, for any integer $n \geq 1$. Our result extends the previously known family of bi-quadratic fields with P\'{o}lya group $\mathbb{Z}/2\mathbb{Z}$ (cf. Theorem 3 in \cite{self-ja}). More precisely, we prove the following theorem.

\begin{theorem}\label{main-TH}
Let $n \geq 1$ be an integer. Then there exist infinitely many totally real bi-quadratic fields $K$ with P\'{o}lya group $Po(K) \simeq (\mathbb{Z}/2\mathbb{Z})^{n}$.
\end{theorem}

Since the P\'{o}lya group is a subgroup of the ideal class group, we immediately have the following corollary.

\begin{corollary}
For an integer $n \geq 1$, there exist infinitely many totally real bi-quadratic fields $K$ with $2$-rank of $Cl_{K}$ being $n$.
\end{corollary}

\section{Preliminaries}

In \cite{zantema}, Zantema made use of Galois cohomology to study $Po(K)$ and its relation with the ramified primes in $K/\mathbb{Q}$, when $K$ is a finite Galois extension of $\mathbb{Q}$. Since the Galois group $G = Gal(K/\mathbb{Q})$ acts on the multiplicative group of units $\mathcal{O}_{K}^{*}$ via the action $(\sigma,\alpha) \mapsto \sigma(\alpha)$, thus endowing $\mathcal{O}_{K}^{*}$ with a $G$-module structure. Then Zantema's result can be stated as follows.

\begin{propn} \cite{zantema} \label{zantema}
Let $K$ be a finite Galois extension of $\mathbb{Q}$ with Galois group $G = Gal(K/\mathbb{Q})$. Let $p_{1},\ldots,p_{m}$  be the rational primes that ramify in $K$. Let the ramification index of $p_{i}$ be $e_{i}$ for each $i \in \{1,\ldots,m\}$. Then there is an exact sequence $$0 \rightarrow H^{1}(G,\mathcal{O}_{K}^{*}) \rightarrow \displaystyle\bigoplus_{i = 1}^{m}\mathbb{Z}/e_{i}\mathbb{Z} \rightarrow Po(K) \rightarrow 0$$ of abelian groups.
\end{propn}

The next two lemma help us understand the structure of the group $H^{1}(G,\mathcal{O}_{K}^{*})$ via its subgroup of $2$-torsion elements.

\begin{lemma} \cite{bennet}\label{lem-1}
Let $K_{1}, K_{2}$ and $K_{3}$ be the quadratic subfields of a bi-quadratic field $K$. Let $H[2]$ be the $2$-torsion subgroup of $H^{1}(G,\mathcal{O}_{K}^{*})$. Then the index of $H[2]$ in $H^{1}(G,\mathcal{O}_{K}^{*})$ is $\leq 2$. The index is $2$ if and only if the rational prime $2$ is totally ramified in $K/\mathbb{Q}$ and there exist $\alpha_{i} \in \mathcal{O}_{K_{i}}$ for each $i = 1,2,3$ with $$N(\alpha_{1}) = N(\alpha_{2}) = N(\alpha_{3}) = \pm {2}.$$
\end{lemma}

\begin{lemma} \cite{zantema}\label{lem-2}
Let $K_{1}, K_{2}$ and $K_{3}$ be the quadratic subfields of a totally real bi-quadratic field $K$. For $i = 1,2,3$, let $\Delta_{i}$ be the square-free part of the discriminant of $K_{i}$. Let $u_{i} = z_{i} + t_{i}\sqrt{\Delta_{i}}$ be a fundamental unit of $\mathcal{O}_{K_{i}}$ with $z_{i} > 0$. Let us define 
\begin{equation*}
a_{i}=
\begin{cases}
    N(u_{i} + 1)   & ~ \text{; if } N(u_{i}) = 1,\\
    1 &  ~ \text{; if } N(u_{i}) = -1.
\end{cases}
\end{equation*}
Let $H[2]$ be the $2$-torsion subgroup of $H^{1}(G,\mathcal{O}_{K}^{*})$. Then $H[2]$ is isomorphic to the subgroup of $\mathbb{Q}^{*}/(\mathbb{Q}^{*})^{2}$ generated by the images of $\Delta_{1}, \Delta_{2}, \Delta_{3}, a_{1}, a_{2}$ and $a_{3}$ in $\mathbb{Q}^{*}/(\mathbb{Q}^{*})^{2}$.
\end{lemma}

The next lemma is very crucial for proving Theorem \ref{main-TH} and we furnish a detailed proof here.

\begin{lemma}\label{our-lemma}
Let $t \geq 2$ be an integer. For every unordered pair $(i,j)$ with $1 \leq i \neq j \leq t$, let $\varepsilon_{ij} \in \{\pm {1}\}$ be given. Then there exist infinitely many prime numbers $p_{1},\ldots ,p_{t}$ with $p_{i} \equiv 1 \pmod {8}$ for each $i$, and the Legendre symbol $\left(\frac{p_{i}}{p_{j}}\right) = \varepsilon_{ij}$ for all $i \neq j$.
\end{lemma}
\begin{proof}
We prove by induction on $t$. We begin with $t = 2$. If $\varepsilon_{12} = 1$, we take $p_{1} = 17$ and $p_{2} \equiv 1 \pmod {8p_{1}}$. On the other hand, if $\varepsilon_{12} = -1$, we take $p_{2} \equiv 3 \pmod {8p_{1}}$. By Dirichlet's theorem for primes in arithmetic progression, we obtain infinitely many choices for $p_{2}$ in both the cases. Therefore, the lemma holds for $t = 2$.


\smallskip

Now, suppose that we have prime numbers $p_{1},\ldots ,p_{t - 1}$ such that $p_{i} \equiv 1 \pmod {8}$ for each $i \in \{1,\ldots ,t - 1\}$ and $\left(\frac{p_{i}}{p_{j}}\right) = \varepsilon_{ij}$ for all $1 \leq i \neq j \leq t - 1$. We wish to find a prime number $p_{t}$ with $\left(\frac{p_{t}}{p_{i}}\right) = \varepsilon_{it}$ for all $1 \leq i \leq t - 1$. For each integer $i \in \{1,\ldots ,t - 1\}$, we choose integers $n_{i}$ such that $1 \leq n_{i} \leq p_{i} - 1$ and $\left(\frac{n_{i}}{p_{i}}\right) = \varepsilon_{it}$. We consider the following system of congruences.

\begin{eqnarray*}
X &\equiv & 1 \pmod {8}\\
X &\equiv & n_{1} \pmod {p_{1}}\\
\vdots \\
X &\equiv & n_{t - 1} \pmod {p_{t - 1}}.
\end{eqnarray*}

By the Chinese remainder theorem, there exists a unique integer, say $x_{0}$, modulo $8p_{1}\ldots p_{t - 1}$ satisfying the above system of congruences. Since $\gcd(x_{0},8p_{1}\ldots p_{t - 1}) = 1$, by Dirichlet's theorem for primes in arithmetic progression, there exist infinitely many prime numbers $\ell$ such that $\ell \equiv x_{0} \pmod {8p_{1}\ldots p_{t - 1}}$. Then $\ell \equiv 1 \pmod {8}$ and $\left(\frac{\ell}{p_{i}}\right) = \left(\frac{x_{0}}{p_{i}}\right) = \left(\frac{n_{i}}{p_{i}}\right) = \varepsilon_{it}$ for each $i \in \{1,\ldots ,t - 1\}$. Taking $p_{t} = \ell$, we conclude that there are infinitely many such prime numbers fulfilling the desired requirements.
\end{proof}

The next lemma, due to Trotter \cite{trotter}, provides a useful criterion for a real quadratic field of a particular type to have a fundamental unit of negative norm.

\begin{lemma} \cite{trotter}\label{trotter}
Let $p_{1},\ldots ,p_{t}$ be prime numbers with $p_{i} \equiv 1 \pmod {4}$ for each $i \in \{1,\ldots ,t\}$. Let $\varepsilon_{ij} = \left(\frac{p_{i}}{p_{j}}\right)$ for $1 \leq i \neq j \leq t$. Then each of the following two conditions is sufficient to ensure that the fundamental unit of the real quadratic field $\mathbb{Q}(\sqrt{p_{1}\cdots p_{t}})$ to have norm $-1$.
\begin{enumerate}
\item $t$ is odd and $\varepsilon_{ij} = -1$ for all $1 \leq i \neq j \leq t$.

\item $t$ is even, $\varepsilon_{12} = -1$, $\varepsilon_{1i} = 1$ for all $i > 2$ and $\varepsilon_{ij} = -1$ for all $2 \leq i \neq j \leq t$.
\end{enumerate}
\end{lemma}

\section{Proof of Theorem \ref{main-TH}}

Let $t \geq 3$ be an odd integer and let $p_{1}\cdots p_{t}$ be prime numbers with $p_{i} \equiv 1 \pmod {8}$ and $\left(\frac{p_{i}}{p_{j}}\right) = -1$ for all $1 \leq i \neq j \leq t$. Lemma \ref{our-lemma} allows us to choose such prime numbers. Let $K = \mathbb{Q}(\sqrt{2},\sqrt{p_{1}\cdots p_{t}})$ with three quadratic subfields $K_{1} = \mathbb{Q}(\sqrt{2}), K_{2} = \mathbb{Q}(\sqrt{p_{1}\cdots p_{t}})$ and $K_{3} = \mathbb{Q}(\sqrt{2p_{1}\cdots p_{t}})$. The only ramified primes in $K/\mathbb{Q}$ are $2$ and $p_{i}$'s. Moreover, since $K = K_{1}K_{2}$ and $K_{1} \cap K_{2} = \mathbb{Q}$, we conclude that the ramification indices of all these primes are $2$. Therefore, by Proposition \ref{zantema}, $H^{1}(G,\mathcal{O}_{K}^{*})$ embeds into $\displaystyle\bigoplus_{i = 1}^{t + 1}\mathbb{Z}/2\mathbb{Z}$, where $G$ is the Galois group $Gal(K/\mathbb{Q})$. Now, we prove that $|H^{1}(G,\mathcal{O}_{K}^{*})| = 4$, which, by Proposition \ref{zantema}, implies that $$|Po(K)| = \frac{\left|\displaystyle\bigoplus_{i = 1}^{t + 1}\mathbb{Z}/2\mathbb{Z}\right|}{|H^{1}(G,\mathcal{O}_{K}^{*})|} = \frac{2^{t + 1}}{4} = 2^{t - 1}.$$

\smallskip

We note that the rational prime $2$ is not totally ramified in $K/\mathbb{Q}$. Therefore, by Lemma \ref{lem-1}, we get that $H[2] = H^{1}(G,\mathcal{O}_{K}^{*})$. Following the notations of Lemma \ref{lem-2}, we get that $\Delta_{1} = 2$, $\Delta_{2} = p_{1}\cdots p_{t}$ and $\Delta_{3} = 2p_{1}\cdots p_{t}$. Using the notation $[x]$ to denote the image of a non-zero element $x \in \mathbb{Q}$ in $\mathbb{Q}^{*}/(\mathbb{Q}^{*})^{2}$, we see that $[\Delta_{1}],[\Delta_{2}],[\Delta_{3}] \in \langle [2], [p_{1}\cdots p_{t}] \rangle$ in $\mathbb{Q}^{*}/(\mathbb{Q}^{*})^{2}$. We note that $1 + \sqrt{2}$ is a fundamental unit of $K_{1}$ with negative norm. Also, by Lemma \ref{trotter}, the fundamental unit of $K_{2}$ has negative norm and consequently, $[a_{1}] \mbox{ and } [a_{2}]$ are trivial in $\mathbb{Q}^{*}/(\mathbb{Q}^{*})^{2}$.

\smallskip

Now, let $u = z + w\sqrt{2p_{1}\cdots p_{t}}$ be a fundamental unit of $K_{3}$. If $N(u) = -1$, then $[a_{3}] \in \langle [2],[p_{1}\cdots p_{t}] \rangle$ and hence $|H^{1}(G,\mathcal{O}_{K}^{*})| = 4$. Otherwise, we assume that $N(u) = 1$. Then by Lemma \ref{lem-2}, we have $a_{3} = N(u + 1) = 2(z + 1)$. Since $z^{2} - 1 = 2w^{2}p_{1}\cdots p_{t}$ is even, we conclude that $z$ is odd and therefore $\gcd(z - 1,z + 1) = 2$. Therefore, $w = 2w_{1}$ for some integer $w_{1}$ and hence
\begin{equation}\label{eq-refer}
\frac{z - 1}{2}\cdot \frac{z + 1}{2} = 2w_{1}^{2}p_{1}\cdots p_{t}.
\end{equation}
Since $\gcd\left(\frac{z - 1}{2},\frac{z + 1}{2}\right) = 1$, from equation \eqref{eq-refer}, we get that $\frac{z - 1}{2} = a^{2}\alpha$ and $\frac{z + 1}{2} = b^{2}\beta$, where $a,b$ are relatively prime integers with $ab = k$ and $\alpha,\beta$ are relatively prime integers with $\alpha\beta = 2p_{1}\cdots p_{t}$. Thus we have $[2(z + 1)] = [4b^{2}\beta] = [\beta] \in \mathbb{Q}^{*}/(\mathbb{Q}^{*})^{2}$. If $\beta = 1 \mbox{ or } 2 \mbox{ or } p_{1}\cdots p_{t} \mbox{ or } 2p_{1}\cdots p_{t}$, then $[\beta] \in \langle [2],[p_{1}\cdots p_{t}] \rangle$. We now prove that other choices cannot occur for $\beta$.

\medskip

\noindent
{\bf Case 1.} $\beta = 2B$, where $B$ is an integer dividing $p_{1}\cdots p_{t}$ and consisting of an odd number of prime divisors. Then for a prime divisor $\ell$ of $\alpha$, from the equation 
\begin{equation}\label{mid-equn}
1 = \frac{z + 1}{2} - \frac{z - 1}{2} = b^{2}\beta - a^{2}\alpha,
\end{equation}
we have $$1 = \left(\frac{1}{\ell}\right) = \left(\frac{2b^{2}B - a^{2}\alpha}{\ell}\right) = \left(\frac{2}{\ell}\right)\left(\frac{B}{\ell}\right) = \left(\frac{B}{\ell}\right) = -1,$$ which is a contradiction. 

\medskip

\noindent
{\bf Case 2.} $\beta = 2B$, where $B$ is an integer dividing $p_{1}\cdots p_{t}$ and consisting of an even number of prime divisors. Then $\alpha$ consists of an odd number of prime divisors, since $t$ is odd. Therefore, for an odd prime divisor $\ell$ of $\beta$, we obtain $$1 = \left(\frac{1}{\ell}\right) = \left(\frac{2b^{2}B - a^{2}\alpha}{\ell}\right) = \left(\frac{-\alpha}{\ell}\right) = \left(\frac{\alpha}{\ell}\right) = -1,$$ which is a contradiction.

\medskip

\noindent
{\bf Case 3.} $\alpha = 2A$, where $A$ is an integer dividing $p_{1}\cdots p_{t}$ and consisting of an even number of prime divisors. Then $\beta$ consists of an odd number of prime factors. Therefore, for an odd prime divisor $\ell$ of $\alpha$, we get $$1 = \left(\frac{1}{\ell}\right) = \left(\frac{b^{2}\beta - 2a^{2}A}{\ell}\right) = \left(\frac{\beta}{\ell}\right) = -1,$$ which is a contradiction.

\medskip

\noindent
{\bf Case 4.} $\alpha = 2A$, where $A$ is an integer dividing $p_{1}\cdots p_{t}$ and consisting of an odd number of prime divisors. Then for a prime divisor $\ell$ of $\beta$, we obtain $$1 = \left(\frac{1}{\ell}\right) = \left(\frac{b^{2}\beta - 2a^{2}A}{\ell}\right) = \left(\frac{-1}{\ell}\right)\left(\frac{2}{\ell}\right)\left(\frac{A}{\ell}\right) = \left(\frac{A}{\ell}\right) = -1,$$ which is a contradiction.

\smallskip

Consequently, we have $[a_{3}] = [2(z + 1)] = [\beta] \in \langle [2],[p_{1}\cdots p_{t}] \rangle$ and hence by Lemma \ref{lem-1} and Lemma \ref{lem-2}, we have $|H^{1}(G,\mathcal{O}_{K}^{*})| = |H[2]| = |\langle [2],[p_{1}\cdots p_{t}] \rangle| = 4$. Hence $|Po(K)| = 2^{t - 1}$. 

\medskip

Now, we assume that $t$ is even. By Lemma \ref{our-lemma}, we can choose prime numbers $p_{1},\ldots ,p_{t}$ such that $p_{i} \equiv 1 \pmod {8}$ for each $i$ and they satisfy condition $(2)$ of Lemma \ref{trotter}. Following the same notation as above, we see that $[\Delta_{1}], [\Delta_{2}], [\Delta_{3}], [a_{1}], [a_{2}] \in \langle [2],[p_{1}\cdots p_{t}] \rangle$. If $N(u) = -1$, then we get that $|H^{1}(G,\mathcal{O}_{K}^{*})| = |H[2]| = |\langle [2],[p_{1}\cdots p_{t}] \rangle| = 4$. We now prove that for $N(u) = 1$ also, we have $[a_{3}] \in \langle [2],[p_{1}\cdots p_{t}] \rangle$. For that, we observe that if $\beta = 1 \mbox{ or } 2 \mbox{ or } p_{1}\cdots p_{t} \mbox{ or } 2p_{1}\cdots p_{t}$, then we are through. We prove that no other possibilities can occur.

\smallskip

\noindent
{\bf Case 1.} $p_{1} \mid \beta$ and $p_{2} \mid \alpha$. Then from equation \eqref{mid-equn}, we get $$1 = \left(\frac{1}{p_{1}}\right) = \left(\frac{-a^{2}\alpha}{p_{1}}\right) = \left(\frac{\alpha}{p_{1}}\right) = \left(\frac{p_{2}}{p_{1}}\right) = -1,$$ by the conditions on $p_{i}$'s. Therefore, this cannot occur.

\medskip

\noindent
{\bf Case 2.} $p_{1} \mid \alpha$ and $p_{2} \mid \beta$. Then again from equation \eqref{mid-equn}, we get $$1 = \left(\frac{1}{p_{1}}\right) = \left(\frac{b^{2}\beta}{p_{1}}\right) = \left(\frac{\beta}{p_{1}}\right) = \left(\frac{p_{2}}{p_{1}}\right) = -1, $$ which is a contradiction

\medskip

\noindent
{\bf Case 3.} $\beta = p_{1}p_{2}B \mbox{ or } 2p_{1}p_{2}B$, where $B \mid p_{1}\cdots p_{t}$ and consists of an even number of odd prime divisors. Then $\alpha$ also comprises with an even number of odd prime factors and therefore from equation \eqref{mid-equn}, we obtain for any odd prime divisor $\ell$ of $\alpha$ that $1 = \left(\frac{1}{\ell}\right) = \left(\frac{2}{\ell}\right)\left(\frac{p_{1}}{\ell}\right)\left(\frac{p_{2}}{\ell}\right)\cdot 1 = -1$, which is a contradiction. Similarly, if $B$ consists of an odd number of odd prime divisors, then so does $\alpha$. Again, from equation \eqref{mid-equn}, we get $1 = \left(\frac{1}{p_{2}}\right) = \left(\frac{-a^{2}\alpha}{p_{2}}\right) = \left(\frac{\alpha}{p_{2}}\right) = -1$, a contradiction.

\medskip

\noindent
{\bf Case 4.} $p_{1}p_{2} \nmid \beta$. Then $p_{1}p_{2} \mid \alpha$ and considering the equation $-1 = a^{2}\alpha - b^{2}\beta$, we are essentially back to Case 3.

\smallskip

Hence we conclude that for $t$ even as well, we have infinitely many bi-quadratic fields $K$ with $|Po(K)| = 2^{t - 1}$. Putting $n = t - 1$ completes the proof of theorem \ref{main-TH}. $\hfill\Box$

\bigskip

{\bf Acknowledgements.} It is a pleasure for the first author to thank Indian Institute of Technology Guwahati for providing the financial support and necessary facilities to carry out this research work. The second author would like to thank MATRICS, SERB for their research grant.

\end{document}